\documentclass{llncs}

\usepackage{amssymb,stmaryrd,mathrsfs,dsfont,mathtools,enumitem,hyperref}
\usepackage[ruled,linesnumbered]{algorithm2e}
\usepackage{algpseudocode}

\usepackage{pgfplots}
\pgfplotsset{compat=1.15}
\usetikzlibrary{arrows}
\usepackage{tikz}
\tikzstyle{vertex}=[circle, draw, inner sep=0pt, minimum size=3pt]
\newcommand{\vertex}{\node[vertex]}

\begin{document}
\frontmatter          
\pagestyle{headings}  

\mainmatter              

\title{Characterization of Split Comparability Graphs}

\titlerunning{Title}  
%

\author{Tithi Dwary \and Khyodeno Mozhui  \and 
K. V. Krishna} 

\authorrunning{Tithi Dwary \and Khyodeno Mozhui \and K. V. Krishna } 

\institute{Indian Institute of Technology Guwahati, India\\
	\email{tithi.dwary@iitg.ac.in};\;\;\; 
	\email{k.mozhui@iitg.ac.in};\;\;\;
	\email{kvk@iitg.ac.in}}

\maketitle              

\begin{abstract}
A split graph is a graph whose vertex set can be partitioned into a clique and an independent set. A split comparability graph is a split graph which is transitively orientable. In this work, we characterize split comparability graphs in terms of vertex labelling. Further, using this characterization, we prove that the permutation-representation number of a split comparability graph is at most three. This gives us an alternative proof of the result in order theory that the dimension of a split order is at most three.
\end{abstract}

\keywords{Split graphs, word-representable graphs, comparability graphs, \textit{prn}, split order.}

\section{Introduction and Preliminaries}

A word is a finite sequence of letters taking from a finite set. A subword $u$ of a word $w$ is a subsequence of $w$, denoted by $u \ll w$. For instance, $aabccb \ll acabbccb$. Let $w$ be a word over a set $X$, and $A \subseteq X$. Then, the subword of $w$ restricted to the elements of $A$ is denoted by $w|_A$. For example, if $w = acabbccb$, then $w|_{\{a, b\}} = aabbb$. For a word $w$, if $w|_{\{a, b\}}$ is of the form $abab \cdots$ or $baba \cdots$, which can be of even or odd length, we say the letters $a$ and $b$ alternate in $w$; otherwise we say $a$ and $b$ do not alternate in $w$. A $k$-uniform word is a word in which every letter occurs exactly $k$ times. For a word $w$, we write $w^R$ to denote its reversal.

A simple graph $G = (V, E)$ is called a word-representable graph, if there exists a word $w$ over $V$ such that for all $a, b \in V$, $\overline{ab} \in E$ if and only if $a$ and $b$ alternate in $w$. A word-representable graph $G$ is said to be $k$-word-representable if there is a $k$-uniform word representing it. In \cite{MR2467435}, It was proved that every word-representable graph is $k$-word-representable, for some $k$. The representation number of a word-representable graph $G$, denoted by $\mathcal{R}(G)$, is defined as the smallest number $k$ such that $G$ is $k$-word-representable. An orientation of a graph is semi-transitive if it is acyclic, and for any directed path $a_1 \rightarrow a_2 \rightarrow \cdots \rightarrow a_k$ either there is no edge between $a_1$ and $a_k$, or $a_i \rightarrow a_j$ is an edge for all $1 \le i < j \le k$. It was proved that a graph is word-representable if and only if it admits a semi-transitive orientation \cite{Halldorsson_2016}. For a detailed introduction to the theory of word-representable graphs, one can refer to the monograph \cite{words&graphs}.

 A word-representable graph $G$ is said to be permutationally representable if there is a word $w$ of the form $p_1p_2 \cdots p_k$ representing $G$, where each $p_i$ is a permutation on the vertices of $G$; in this case $G$ is called a permutationally $k$-representable graph. Moreover, we say $w$ permutationally represents $G$.  The permutation-representation number (in short, \textit{prn}) of $G$, denoted by $\mathcal{R}^p(G)$, is the smallest number $k$ such that $G$ is permutationally $k$-representable. It was shown in \cite{perkinsemigroup} that a graph is permutationally representable if and only if it is a comparability graph - a graph which admits a transitive orientation. Recall that an orientation of a graph $G = (V, E)$ is transitive if $\overrightarrow{ab} \in E$, $\overrightarrow{bc} \in E$, then $\overrightarrow{ac} \in E$, for all $a, b, c \in V$. Note that a transitive orientation is also a semi-transitive orientation, but not conversely. It was reconciled in \cite{khyodeno2} that if $G$ is a comparability graph, then $\mathcal{R}^p(G)$ is precisely the dimension of any induced partially ordered set (poset) of $G$. 
 
 The class of graphs with representation number at most two is characterized as the class of circle graphs \cite{Hallsorsson_2011} and the class of graphs with \textit{prn} at most two is the class of permutation graphs \cite{Gallaipaper}. In general, the problems of determining the representation number of a word-representable graph, and the \textit{prn} of a comparability graph are computationally hard \cite{Hallsorsson_2011,yannakakis1982complexity}.

A graph $G = (V, E)$ is a split graph if the vertex set $V$ can be partitioned as $I \cup C$, where $I$ induces an independent set, and $C$ induces a clique in $G$. In this paper, we consider the clique $C$ for a split graph $G = (I \cup C, E)$ to be inclusion wise maximal, i.e., no vertices of $I$ is adjacent to all vertices of $C$. The class of  word-representable split graphs was characterized in \cite{Kitaev_2021} using semi-transitive orientation. Further, a forbidden induced subgraph characterization for the class of split comparability graphs (split graphs which are transitively orientable) can be found in \cite{Golumbicbook_2004}. Recall from \cite{split_orders2004} that a partial order is called a split order if the corresponding comparability graph is a split graph. It was proved in \cite{split_orders2004} that the dimension of a split order is at most three, and the bound is tight. In the following, we recall some relevant results known for the class of split graphs restricted to permutation graphs, comparability graphs and word-representable graphs.

\begin{theorem}[\cite{Split_circle_graphs}]\label{Split_permu_graph}
	Let $G$ be a split graph. Then, $G$ is a permutation graph if and only if $G$ is a $\mathcal{B}$-free graph, where $\mathcal{B}$ is the class of graphs given in Fig. \ref{fig7}. 
\end{theorem}

\begin{figure}[t]
	\centering
	\begin{minipage}{.25\textwidth}
		\centering
		\[\begin{tikzpicture}[scale=0.6]

			\vertex (1) at (0,0) [fill=black,label=left:$ $] {};
			\vertex (2) at (1.5,0) [fill=black,label=left:$ $] {};
			\vertex (3) at (0,1) [fill=black,label=left:$ $] {};	
			\vertex (4) at (1.5,1) [fill=black,label=left:$ $] {};
			\vertex (5) at (0.75,2) [fill=black,label=left:$ $] {};	
			\vertex (6) at (0.75,3) [fill=black,label=left:$ $] {};

			\path 
			(1) edge (3)
			(2) edge (4)
			(3) edge (5)
			(4) edge (5)
			(5) edge (6)
			(3) edge (4);

		\end{tikzpicture}\] 
		$B_1$	
	\end{minipage}%
	\begin{minipage}{.25\textwidth}
		\centering
		
		\[\begin{tikzpicture}[scale=0.6]

			\vertex (1) at (0,0) [fill=black,label=left:$ $] {};
			\vertex (2) at (1.5,0) [fill=black,label=left:$ $] {};
			\vertex (3) at (3,0) [fill=black,label=left:$ $] {};	
			\vertex (4) at (0.75,1) [fill=black,label=left:$ $] {};
			\vertex (5) at (2.25,1) [fill=black,label=left:$ $] {};	
			\vertex (6) at (1.5,2) [fill=black,label=left:$ $] {};

			\path
			(1) edge (2)
			(2) edge (3)
			(1) edge (4)
			(2) edge (4)
			(2) edge (5)
			(3) edge (5)
			(4) edge (5)
			(4) edge (6)
			(5) edge (6); 
			
		\end{tikzpicture}\]
		$B_2$
	\end{minipage}%
	\begin{minipage}{.25\textwidth}
		\centering
		
		\[\begin{tikzpicture}[scale=0.6]
			
			\vertex (1) at (0,0) [fill=black,label=left:$ $] {};
			\vertex (2) at (1.5,0) [fill=black,label=left:$ $] {};
			\vertex (3) at (0,1) [fill=black,label=left:$ $] {};	
			\vertex (4) at (1.5,1) [fill=black,label=left:$ $] {};
			\vertex (5) at (0.75,2) [fill=black,label=left:$ $] {};	
			\vertex (6) at (-0.75,2) [fill=black,label=left:$ $] {};	
			\vertex (7) at (2.25,2) [fill=black,label=left:$ $] {};

			\path 
			(1) edge (3)
			(2) edge (4)
			(3) edge (5)
			(4) edge (5)
			(5) edge (6)
			(3) edge (4)
			(3) edge (6)
			(5) edge (7)
			(4) edge (7);

		\end{tikzpicture}\]
		$B_3$
	\end{minipage}%
	\begin{minipage}{.25\textwidth}
		\centering
		
		\[\begin{tikzpicture}[scale=0.4]
			
			\vertex (1) at (0,0) [fill=black,label=below:$ $] {};
			\vertex (2) at (2,0) [fill=black,label=below:$ $] {};
			\vertex (3) at (2,2) [fill=black,label=right:$ $] {};
			\vertex (4) at (0,2) [fill=black,label=left:$ $] {};
			\vertex (5) at (1,3.2) [fill=black,label=above:$ $] {};
			\vertex (6) at (3.2,1) [fill=black,label=right:$ $] {};
			\vertex (7) at (-1.2,1) [fill=black,label=left:$ $] {};
			
			\path 
			(1) edge (2)
			(1) edge (7)
			(1) edge (4)
			(1) edge (3)
			(2) edge (3)
			(2) edge (4)
			(2) edge (6)
			(3) edge (4)
			(3) edge (5)
			(3) edge (6)
			(4) edge (5)
			(4) edge (7);
			
		\end{tikzpicture}\]
		$B_4$
	\end{minipage}%
	\caption{The family of graphs $\mathcal{B}$}
	\label{fig7}	
\end{figure}
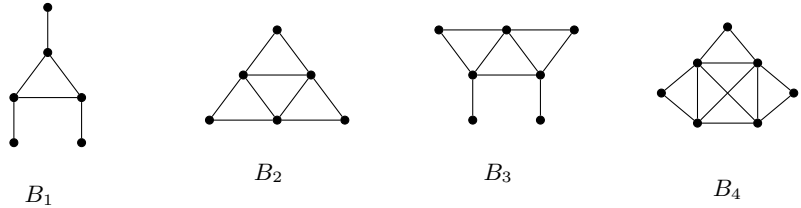

\begin{theorem}[\cite{Golumbicbook_2004}] \label{Split_comp_graph}
	Let $G$ be a split graph. Then, $G$ is a comparability graph if and only if $G$ contains no induced subgraph isomorphic to $B_1$, $B_2$, or $B_3$ (depicted in Fig. \ref{fig7}).
\end{theorem}

\begin{theorem}[\cite{Kitaev_2021}]\label{th_7}
	Let $G = (I \cup C, E)$ be a split graph with $|C| = k$, and $S$ be a semi-transitive orientation of $G$ with the longest directed path $P: c_1 \rightarrow \cdots \rightarrow c_k$ in $C$. Then, the vertices of $I$ can be subdivided into the following three, possibly empty, disjoint types. For $a \in I$, 
	
	\begin{enumerate}[label=\rm (\alph*)]
		\item \label{point_III} $N(a) = \{c_1, \ldots, c_m\} \cup \{c_n, \ldots, c_k\}$, for some $1 \le m < n \le k$, such that $\overrightarrow{c_ia} \in S$ for all $1 \le i \le m$ and $\overrightarrow{ac_i} \in S$ for all $n \le i \le k$;
		
		\item \label{point_II} $N(a) = \{c_l, \ldots, c_r\}$, for some $1 \le l \le r \le k$, such that $\overrightarrow{c_ia} \in S$ for all $l \le i \le r$ $($$a$ is a sink$)$; or
		
		\item \label{point_I} $N(a) = \{c_l, \ldots, c_r\}$, for some $1 \le l \le r \le k$, such that $\overrightarrow{ac_i} \in S$ for all $l \le i \le r$ $($$a$ is a source$)$.	
	\end{enumerate}
\end{theorem}

For any two integers $a \le b$, $[a, b]$ in Theorem \ref{Word_split_graph} denotes the set of integers $\{a, a+1, \ldots, b\}$.

\begin{theorem}[\cite{Kitaev_2021,Kitaev_2024}]\label{Word_split_graph}
	Let $G = (I \cup C, E)$ be a split graph. Then, $G$ is semi-transitively orientable if and only if the vertices of $C$ can be labeled from $1$ to $k = |C|$ in such a way that for each $a, b \in I$ the following holds.
	\begin{enumerate}[label=\rm (\roman*)]
		\item \label{point_1} Either $N(a) = [1, m] \cup [n, k]$, for $m < n$, or $N(a) = [l, r]$, for $l \le r$.
		\item \label{point_2} If $N(a) = [1, m] \cup [n, k]$ and $N(b) = [l, r]$, for $m < n$ and $l \le r$, then $l > m$ or $r < n$.
		\item \label{point_3} If $N(a) = [1, m] \cup [n, k]$ and $N(b) = [1, m'] \cup [n', k]$, for $m < n$ and $m' < n'$, then $m' < n$ and $m < n'$.
	\end{enumerate}
\end{theorem}

\section{Our Contributions}

This work is a continuation of the work in \cite{tithi_split}, in which the representation number of word-representable split graphs was obtained through an algorithmic procedure using the labelling given in Theorem \ref{Word_split_graph}. In this work, we extend the aforesaid algorithm to obtain the \textit{prn} for split comparability graphs. For which, first we characterize the class of split comparability graphs in terms of the vertex labelling. Using this characterization, we devise an algorithmic procedure to construct a $3$-uniform word permutationally representing a split comparability graph. This shows that the permutation-representation number of a split comparability graph is at most three. Accordingly, we obtain an alternative proof for the result \cite[Theorem 22]{split_orders2004} in order theory that the dimension of a split order is at most three. Additionally, we determine the class of split comparability graphs having permutation-representation number exactly three.

\section{Characterization}

Let $G = (I \cup C, E)$ be a split comparability graph with $|C| = k$. Suppose $D$ is a transitive orientation of $G$ with the longest directed path in $C$, say $P: c_1 \rightarrow \cdots \rightarrow c_k$. We extend Theorem \ref{th_7} for the transitive orientation $D$ and give a classification of elements of $I$ in the following lemma.

\begin{lemma}\label{neighbors}
The set $I$ is a disjoint union of $I_1$, $I_2$ and $I_3$ defined by the following:  
	\begin{enumerate}[label=\rm (\roman*)]
		\item $I_1$ consists of $a \in I$ such that $N(a) = \{c_1, \ldots, c_m\} \cup \{c_n, \ldots, c_k\}$, for some $1 \le m < n \le k$, with $\overrightarrow{c_ia} \in D$ for all $1 \le i \le m$ and $\overrightarrow{ac_i} \in D$ for all $n \le i \le k$;
			
		\item $I_2$ consists of $a \in I$ such that  $N(a) = \{c_1, \ldots, c_r\}$, for some $r < k$, with $\overrightarrow{c_ia} \in D$ for all $1 \le i \le r$ $(a$ is a sink$)$; or
		
		\item $I_3$ consists of $a \in I$ such that  $N(a) = \{c_l, \ldots, c_k\}$, for some $l > 1$, with $\overrightarrow{ac_i} \in D$ for all $l \le i \le k$ $(a$ is a source$)$.
	\end{enumerate}
\end{lemma}

\begin{proof}
	Since every transitive orientation is also a semi-transitive orientation, if $a \in I$, then $a$ belongs to one of the three types defined in Theorem \ref{th_7}. Note that the elements of type (a) are precisely in $I_1$. Suppose $a$ is of type \ref{point_II} so that $N(a) = \{c_l, \ldots, c_r\}$, for some $1 \le l \le r \le k$, such that $\overrightarrow{c_ia} \in D$ for all $l \le i \le r$. If $l > 1$, then we have $\overrightarrow{c_1c_l} \in D$ and $\overrightarrow{c_la} \in D$,  but $\overrightarrow{c_1a} \notin D$; a contradiction to $D$ is a transitive orientation of $G$. Thus, we must have $l = 1$. Further, as $C$ is a maximal clique in $G$, we have $r < k$; otherwise, the clique $C$ can be enlarged by including $a$. Hence, if $a$ is of type (b), then $a \in I_2$. Similarly, if $a$ is of type \ref{point_I}, we can prove that $a \in I_3$.
  \qed
\end{proof}

\begin{lemma}\label{cor_2}
	For $a, b \in I$, if $a \in I_2$ and $b \in I_3$, then $N(a) \cap N(b) = \varnothing$. That is, if $N(a) = \{c_1, \ldots, c_r\}$, for some $r < k$, and $N(b) = \{c_l, \ldots, c_k\}$, for some $l > 1$, then  $r < l$.
\end{lemma}

\begin{proof}
	Note that $a$ and $b$ are sink and source, respectively, in $D$. On the contrary, suppose that $r \ge l$ so that we have $\overrightarrow{bc_r} \in D$ and $\overrightarrow{c_ra} \in D$. Since $D$ is a transitive orientation, we must have $\overrightarrow{ba} \in D$. This contradicts $\overline{ab} \notin E$, as $a, b \in I$. \qed
\end{proof}

\begin{theorem}\label{coro_3}
	Let $G = (I \cup C, E)$ be a split graph. Then, $G$ is transitively orientable if and only if the vertices of $C$ can be labeled from $1$ to $k = |C|$ such that the following properties $({\rm i})-({\rm v})$ hold: For $a, b \in I$,
	\begin{enumerate}[label=\rm (\roman*)]
		\item \label{pt_1} The neighborhood $N(a)$ has one of the following forms: $[1, m] \cup [n, k]$ for $m < n$, $[1, r]$ for $r < k$, or $[l, k]$ for $l > 1$.
		
		\item \label{pt_2} If $N(a) = [1, r]$ and $N(b) = [l, k]$, for $r < k$ and $l > 1$, then $r < l$.
		
		\item \label{pt_3} If $N(a) = [1, m] \cup [n, k]$ and $N(b) = [1, r]$,  for $r < k$ and $m < n$, then $r < n$.
		
		\item \label{pt_4} If $N(a) = [1, m] \cup [n, k]$ and $N(b) = [l, k]$, for $l > 1$ and $m < n$, then $m < l$.
		
		\item \label{pt_5} If $N(a) = [1, m] \cup [n, k]$ and  $N(b) = [1, m'] \cup [n', k]$, for $m < n$ and $m' < n'$, then $m < n'$ and $m' < n$.
	\end{enumerate}
\end{theorem}

\begin{proof}
	Let $D$ be a transitive orientation of $G$ with the longest directed path $c_1 \rightarrow c_2 \rightarrow \cdots \rightarrow c_k$ in $C$. For each $1 \le i \le k$, we label the vertex $c_i$ of $C$ as $i$. Then, in view of lemmas \ref{neighbors} and \ref{cor_2}, points \ref{pt_1} and \ref{pt_2} hold. Further, since a transitive orientation is also a semi-transitive orientation, points \ref{pt_3}, \ref{pt_4}, and \ref{pt_5} follow from Theorem \ref{Word_split_graph}\ref{point_2}\&\ref{point_3}.
	
	Conversely, suppose the vertices of $C$ are labeled from $1$ to $k$ such that it satisfies the properties (i) -- (v). We now define the following sets: 
	
	$A_1 = \{a \in I \mid N(a) = [1, m] \cup [n, k], \ \text{for some} \ m < n\}$, 
	
	$A_2 = \{a \in I \mid N(a) = [1, r], \ \text{for some} \ r < k\}$, and 
	
	$A_3 = \{a \in I \mid N(a) = [l, k], \ \text{for some} \ l > 1\}$. \\ Further, we define an orientation $D$ of $G$ as per the following:
	\begin{itemize}
		\item $\overrightarrow{ij} \in D$, for all $1 \le i < j \le k$.
		
		\item If $a \in A_1$ with $N(a) = [1, m] \cup [n, k]$, for some $m < n$, then $\overrightarrow{ia} \in D$, for all $1 \le i \le m$, and $\overrightarrow{aj} \in D$, for all $n \le j \le k$.
		
		\item If $a \in A_2$, then make it a sink in $D$.
		
		\item If $a \in A_3$, then make it a source in $D$.	
	\end{itemize}
	We show that $D$ is a transitive orientation. On the contrary, suppose there exist vertices $a, b, c$ of $G$ such that $\overrightarrow{ab} \in D$, $\overrightarrow{bc} \in D$, but $\overrightarrow{ac} \notin D$. Observe that at least one of $a$ and $c$ belong to $I$. If both $a, c \in C$, we must have $\overrightarrow{ca} \in D$, which implies that $a, b, c$ form a directed cycle. However, in view of Theorem \ref{Word_split_graph}, $D$ is indeed a semi-transitive orientation and thus, $D$ is acyclic. Hence, $I$ has at least one of $a$ and $c$. Accordingly, we have $b \in C$ so that $b = s$, for some $1 \le s \le k$.
	
	In the following, we will show that none of $a$ and $c$ belong to $I$, leading to a contradiction to our assumption. 
	
	\paragraph{Claim: $a, c \not\in A_3$.}  Since the vertices of $A_3$ are sources in $D$, note that $c \notin A_3$. Suppose $a \in A_3$ and $N(a) = [l, k]$, for some $l < k$. Note that $l \le s$, as $s \in N(a)$. We first show that $c \not\in C \cup A_2$ as well. 
	
	If $c \in C$, then $c = t$, for some $t > s$ (as $\overrightarrow{sc} = \overrightarrow{bc}   \in D$). But as $a \in A_3$ and $\overrightarrow{as} \in D$, by the property of $A_3$, we have $\overrightarrow{at} = \overrightarrow{ac} \in D$, but $\overrightarrow{ac} \not\in D$. If $c \in A_2$, then $N(c) = [1, r]$, for some $r < k$. Since $N(a) = [l, k]$, by point \ref{pt_2}, we have $r < l$. But, since $l \le s$, we have $r < s$; which is not possible as $s \in [1, r]$. 
	
	Thus, $c \in A_1$. Accordingly, let $N(c) = [1, m] \cup [n, k]$, for some $m < n$. Since $N(a) = [l, k]$, by property \ref{pt_4}, we have $m < l$. Now, as $\overrightarrow{sc} \in D$, we must have $s \le m$ (by property of $A_1$). But, as $m < l$, we get $s < l$, a contradiction to $l \le s$. Thus, none of $a$ and $c$ belong to $A_3$.
	
	\paragraph{Claim: $a, c \not\in A_2$.} Since each vertex of $A_2$ is a sink, we have $a \notin A_2$. Suppose $c \in A_2$ and $N(c) = [1, r]$, for some $r < k$. As $s \in N(c)$, note that $s \le r$. 
	
	If $a \in C$, then $a = t$, for some $t < s$, as $\overrightarrow{as} \in D$. But, since $c \in A_2$ and $\overrightarrow{sc} \in D$, by property of $A_2$, we have $\overrightarrow{tc} \ (= \overrightarrow{ac}) \in D$; which is not possible. Hence, $a \in I$. However, since $a \notin A_2 \cup A_3$ (including the conclusion of previous case), we must have $a \in A_1$. 
	
	Let $N(a) = [1, m] \cup [n, k]$, for some $m < n$. Since $N(c) = [1, r]$, by point \ref{pt_3}, we have $r < n$. As $\overrightarrow{as} \in D$, by property of $A_1$, we have $n \le s$. But, as $r < n$, we have $r < s$, a contradiction to $s \le r$. Thus, none of $a$ and $c$ belong to $A_2$.
	
	\paragraph{Claim: $a, c \not\in A_1$.} Since $a, c \not\in A_2 \cup A_3$, we consider only the following cases.
	\begin{itemize}
		\item Both $a$ and $c$ belong to $A_1$: Let $N(a) = [1, m] \cup [n, k]$, for some $m < n$, and $N(c) = [1, m'] \cup [n', k]$, for some $m' < n'$. Then, by property \ref{pt_5}, we have $m < n'$ and $m' < n$. Since $\overrightarrow{as} \in D$ and $\overrightarrow{sc} \in D$, by property of $A_1$, we must have $ n \le s \le m'$ so that $n \le m'$, a contradiction to $m' < n$.
		
		\item $a \in A_1$ and $c \in C$: Let $N(a) = [1, m] \cup [n, k]$, for some $m < n$. As $\overrightarrow{sc} \in D$, we have $c = t$, for some $t > s$.	Moreover, as $\overrightarrow{as} \in D$, we have $n \le s < t \le k$ so that $\overrightarrow{at} = \overrightarrow{ac} \in D$, a contradiction.
		
		\item $c \in A_1$ and $a \in C$: Let $N(c) = [1, m] \cup [n, k]$, for some $m < n$. As $\overrightarrow{as} \in D$, we have $a = t$ with $t < s$. Further, as $\overrightarrow{sc} \in D$, we have $1 \le t < s \le m$ so that we have $\overrightarrow{tc}  =\overrightarrow{ac} \in D$, a contradiction.
	\end{itemize}
	Hence, for any three vertices $a, b, c$ of $G$, if $\overrightarrow{ab} \in D$ and $\overrightarrow{bc} \in D$, then $\overrightarrow{ac} \in D$ so that $D$ is a transitive orientation. \qed
\end{proof}

\section{Permutation-Representation Number}

Let $G = (I \cup C, E)$ be a split comparability graph. Then, by Theorem \ref{coro_3}, the vertices of $C$ can be labeled from $1$ to $k = |C|$ such that it satisfies all the properties given in Theorem \ref{coro_3}. Throughout this section, we consider the aforementioned labelling of the vertices of $C$. We now consider the following sets:
\begin{align*}
	A_1 & = \{a \in I \mid N(a) = [1, m] \cup [n, k], \ \text{for some} \ m < n\} \\
	A_2 & = \{a \in I \mid N(a) = [1, r], \ \text{for some} \ r < k\} \\
	A_3 & = \{a \in I \mid N(a) = [l, k], \ \text{for some} \ l > 1\}
\end{align*}
Then, from Theorem \ref{coro_3}, we have $A_1, A_2$ and $A_3$ are mutually disjoint, and $I = A_1 \cup A_2 \cup A_3$. For $a \in A_1$, let $N(a) = [1, m_a] \cup [n_a, k]$, for $a \in A_2$, let $N(a) = [1, r_a]$, and for $a \in A_3$, let $N(a) = [l_a, k]$. By Theorem \ref{coro_3}\ref{pt_4}, since $m_a < l_b$ for any $a \in A_1$ and $b \in A_3$, we have the following remark.

\begin{remark}\label{A_1A_3}
	Note that d = $\max\{m_a \mid a \in A_1\}$, and $d < \min\{l_a \mid a \in A_3\}$.
\end{remark}

Given a split graph $G = (I \cup C, E)$ with a transitive orientation (and hence, a labelling of the vertices of $C$), Algorithm \ref{Algo_2} constructs three permutations on $I \cup C$ whose concatenation represents $G$. In what follows, the word $z$ refers to the output of Algorithm \ref{Algo_2}.

\begin{algorithm}[tb]\label{Algo_2}
	\caption{Constructing a $3$-uniform word permutationally representing a split comparability graph.}
	\label{algo-2}
	\KwIn{A split comparability graph $G = (I \cup C, E)$ with $1, 2, \ldots, k$ as  labels of the vertices of $C$.}
	\KwOut{A $3$-uniform word $z$ permutationally representing $G$.}
	
	Initialize $q_1$, $q_2$ and $q_3$ with the permutation $12 \cdots k$ and update them with the vertices of $I$, as per the following.\\
	Initialize $d = 1$.\\
	\For{$a \in I$}{\If{$a \in A_1$}{ Let $N(a) = [1, m] \cup [n, k]$.\\
			\If{$m > d$}{Assign $d = m$.}
			
			Replace $m$ in $q_1$ with $ma$.\\
			Replace $n$ in $q_2$ with $an$.}
		\ElseIf{$a \in A_2$}{ Let $N(a) = [1, r]$.\\
				
				Replace $r$ in $q_2$ with $ra$.\\
				Update $q_3$ with $q_3a$.}	
		\Else{Let $N(a) = [l, k]$.\\
			Replace $l$ in $q_2$ with $al$.\\
			Replace $l$ in $q_3$ with $al$.}
	}
	Update $q_1$ with $(q_2|_{A_3})^Rq_1 (q_2|_{A_2})^R$.\\
	Replace $d$ in $q_3$ with $d (q_1|_{A_1})^R$.\\
	\Return{{\rm the word} $q_1q_2q_3$}	
\end{algorithm}

\begin{remark}\label{position_C_Alg2}
	In Algorithm \ref{Algo_2}, note that the order of the vertices of $C$, i.e., $12 \cdots k$, in the permutations $q_1, q_2$, and $q_3$ do not alter when they are updated. Hence, we have  $z|_{C}  = 12 \cdots k12 \cdots k12 \cdots k$.
\end{remark}

\begin{remark}\label{remark_8}
	Note that, the set $A \subseteq I$ defined in Section 3 of \cite{tithi_split} and the set $A_1$ are identical for a split comparability graph $G$. Moreover, the lines related to insertion of the elements of $A$ in $p_1$, $p_2$ and $p_3$ in Algorithm 1 of \cite{tithi_split} are exactly the same as the lines on insertion of the elements of $A_1$ in $q_1$, $q_2$ and $q_3$ in Algorithm \ref{Algo_2}. Hence, the positions of the vertices of $A_1$ in $q_1, q_2$, and $q_3$ with respect to the vertices of $C$ are same as the positions of the vertices of $A$ in $p_1, p_2, p_3$ with respect to the vertices of $C$, respectively, i.e., $p_1|_{A \cup C} = q_1|_{A_1 \cup C}$, $p_2|_{A \cup C} = q_2|_{A_1 \cup C}$ and $p_3|_{A \cup C} = q_3|_{A_1 \cup C}$.
\end{remark}

\begin{lemma}\label{lemma_9}
	For $a, b \in V$, if $a$ and $b$ are adjacent in $G$, then $a$ and $b$ alternate in the word $z$.
\end{lemma}

\begin{proof}
	As $I$ is an independent set in $G$, note that at least one of $a$ and $b$ is in $C$. Suppose both $a, b \in C$. As the vertices of $C$ are labeled from $1$ to $k$, we have $a = i$ and $b = j$, for some $1 \le i, j \le k$. Thus, in view of Remark \ref{position_C_Alg2}, we have $z|_{\{a, b\}} = ababab$, if $i < j$, or $z|_{\{a, b\}} = bababa$, if $j < i$. Hence, in this case, $a$ alternates with $b$ in $z$. If $a \in I$ and $b \in C$, then we consider the following cases:
	\begin{itemize}
		\item Case 1: $a \in A_1$. In view of Remark \ref{remark_8} and Case 1 in the proof of Lemma 2 in \cite{tithi_split}, we conclude that $a$ and $b$ alternate in the word $z$.

		\item Case 2: $a \in A_2$. Note that $b \in N(a) = [1, r_a]$, for $r_a < k$. Thus, by line $1$ of Algorithm \ref{Algo_2}, we have $br_a \ll q_1$, $br_a \ll q_2$, and $br_a \ll q_3$. As $a \in A_2$, by replacing $r_a$ in $q_2$ with $r_aa$ (see line 14) and updating $q_3$ with $q_3a$ (see line 15), we have $br_aa \ll q_2$ and $br_aa \ll q_3$, respectively. Further, by line 23, we have $br_aa \ll q_1$. Thus, $ba \ll q_1$, $ba \ll q_2$, and $ba \ll q_3$ so that $a$ and $b$ alternate in the word $z$.
		
		\item Case 3: $a \in A_3$. Note that $b \in N(a) = [l_a, k]$, for $l_a > 1$. Thus, we have $l_ab \ll q_1$, $l_ab \ll q_2$, and $l_ab \ll q_3$. As $a \in A_3$, replacing $l_a$ in $q_2$ with $al_a$ (see line 19) and replacing $l_a$ in $q_3$ with $al_a$ (see line 20), we have $al_ab \ll q_2$ and $al_ab \ll q_3$, respectively. Further, by line 23, we have $al_ab \ll q_1$. Hence, $ab \ll q_1$, $ab \ll q_2$, and $ab \ll q_3$ so that $a$ and $b$ alternate in the word $z$.
	\end{itemize}
	\qed 
\end{proof}

\begin{lemma}\label{lemma_10}
	If $a, b \in I$, then $a$ and $b$ do not alternate in $z$.
\end{lemma}

\begin{proof}
	Since $I = A_1 \cup A_2 \cup A_3$, we consider the following cases.
	\begin{itemize}
		\item Case 1: $a, b \in A_1$. In view of Remark \ref{remark_8} and Case 1 in the proof of Lemma 3 in \cite{tithi_split}, we conclude that $a$ and $b$ do not alternate in the word $z$.
		
		\item Case 2: $a, b \in A_2$. Without loss of generality, suppose $ab \ll q_2$. Then, by line 23 of Algorithm \ref{algo-2}, we have $ba \ll q_1$. Thus, $a$ and $b$ do not alternate in the word $z$.
		
		\item Case 3: $a, b \in A_3$. Without loss of generality, suppose $ab \ll q_2$. Then, $ba \ll q_1$ (by line 23). Thus, $a$ and $b$ do not alternate in the word $z$.
		
		\item Case 4: $a \in A_1$ and $b \in A_2$. By line 23 of Algorithm \ref{algo-2}, we have $ab \ll q_1$. Note that $N(a) = [1, m_a] \cup [n_a, k]$, for $m_a < n_a$, and $N(b) = [1, r_b]$, for $r_b < k$. Then, by Theorem \ref{coro_3} \ref{pt_3}, we have $r_b < n_a$ so that $r_bn_a \ll q_2$. Thus, by replacing $r_b$ in $q_2$ with $r_bb$ (see line 14) and replacing $n_a$ in $q_2$ with $an_a$ (see line 10), we have $r_bban_a \ll q_2$ so that $ba \ll q_2$. Hence, $a$ and $b$ do not alternate in the word $z$.
		
		\item Case 5: $a \in A_1$ and $b \in A_3$. Then, by line 23 of Algorithm \ref{algo-2}, we have $ba \ll q_1$. Note that $N(a) = [1, m_a] \cup [n_a, k]$, for $m_a < n_a$, and $N(b) = [l_b, k]$, for $l_b > 1$. Then, in view of Remark \ref{A_1A_3}, we have $m_a \le d < l_b$ so that $m_adl_b \ll q_3$. Now, by replacing $d$ in $q_3$ with $da$ (see line 24) and by replacing $l_b$ in $q_3$ with $bl_b$ (see line 20), we have $ab \ll m_adabl_b \ll q_3$. Hence, $ba \ll q_1$ and $ab \ll q_3$ so that $a$ and $b$ do not alternate in the word $z$.
		
		\item Case 6: $a \in A_2$ and $b \in A_3$. Then, by line 23 of Algorithm \ref{algo-2}, we have $ba \ll q_1$. Note that $N(a) = [1, r_a]$, for $r_a < k$, and $N(b) = [l_b, k]$, for $l_b > 1$. Then, from Theorem \ref{coro_3} \ref{pt_2}, we have $r_a < l_b$ so that $r_al_b \ll q_2$. Now, by replacing $r_a$ in $q_2$ with $r_aa$ (see line 14) and by replacing $l_b$ in $q_2$ with $bl_b$ (see line 19), we have $ab \ll r_aabl_b \ll q_2$. Hence, $ba \ll q_1$ and $ab \ll q_2$ so that $a$ and $b$ do not alternate in the word $z$.
	\end{itemize}
Thus, in any case, $a$ and $b$ do not alternate in $z$. \qed	
\end{proof}

\begin{lemma}\label{lemma_11}
	For $a \in I$ and $b \in C$, if $a$ and $b$ are not adjacent in $G$, then $a$ and $b$ do not alternate in $z$.
\end{lemma}

\begin{proof}
	As $I = A_1 \cup A_2 \cup A_3$, we consider the following cases.
	\begin{itemize}
		\item Case 1: $a \in A_1$. In view of Remark \ref{remark_8} and Case 1 in the proof of Lemma 4 in \cite{tithi_split}, we conclude that $a$ and $b$ do not alternate in the word $z$.
		
		\item Case 2: $a \in A_2$. From line 23 of Algorithm \ref{algo-2}, we have $ba \ll q_1$. Note that $N(a) = [1, r_a]$, for $r_a < k$. Since $b \notin N(a)$, we have $r_a < b \le k$ so that $r_ab \ll q_2$. Now, as per line 14 of Algorithm \ref{algo-2}, replacing $r_a$ in $q_2$ with $r_aa$, we have $r_aab \ll q_2$. Hence, $ba \ll q_1$ and $ab \ll q_2$ so that $a$ and $b$ do not alternate in the word $z$.
		
		\item Case 3: $a \in A_3$. From line 23 of Algorithm \ref{algo-2}, we have $ab \ll q_1$. Note that $N(a) = [l_a, k]$, for $l_a > 1$. Since $b \notin N(a)$, we have $1 \le b < l_a$ so that $bl_a \ll q_2$. Now, as per line 19 of Algorithm \ref{algo-2}, replacing $l_a$ in $q_2$ with $al_a$, we have $bal_a \ll q_2$. Thus, $ab \ll q_1$ and $ba \ll q_2$. Hence, $a$ and $b$ do not alternate in the word $z$.
	\end{itemize}
Hence, in any case, $a$ and $b$ do not alternate in the word $z$. \qed
\end{proof}

\begin{theorem}\label{3-per-rep}
	Let $G = (I \cup C, E)$ be a split comparability graph. Then, $\mathcal{R}^p(G) \le 3$. Moreover, $\mathcal{R}^p(G) = 3$ if and only if $G$ contains $B_4$ (given in Fig. \ref{fig7}) as an induced subgraph. 
\end{theorem}

\begin{proof}
	Note that, if $a$ and $b$ are non-adjacent vertices of $G$, then either both $a, b \in I$ or one of them belongs to $I$. Thus, by lemmas \ref{lemma_9}, \ref{lemma_10}, and \ref{lemma_11}, the word $q_1q_2q_3$ permutationally represents $G$ so that $\mathcal{R}^p(G) \le 3$. 
	
	Note that $B_4$ is a split graph and it contains no induced subgraph isomorphic to $B_1, B_2$ or $B_3$ (refer Fig. \ref{fig7}). Thus, by Theorem \ref{Split_comp_graph}, $B_4$ is a comparability graph. Hence, $\mathcal{R}^p(B_4) \le 3$.  Also, from Theorem \ref{Split_permu_graph}, $B_4$ is not a permutation graph so that $\mathcal{R}^p(B_4) \ge 3$. Hence, $\mathcal{R}^p(B_4) = 3$. If $B_4$ is an induced subgraph of $G$, then clearly $\mathcal{R}^p(G) = 3$.
	
	Conversely, suppose $\mathcal{R}^p(G) = 3$ so that $G$ is not a permutation graph. Thus, by Theorem \ref{Split_permu_graph}, $G$ must contain one of the graphs from the family $\mathcal{B}$ (see Fig. \ref{fig7}) as an induced subgraph. In view of Theorem \ref{Split_comp_graph}, since $B_4$ is the only comparability graph in the family $\mathcal{B}$, the graph $G$ must contain $B_4$ as an induced subgraph.  \qed   
\end{proof}

In view of the connection between the \textit{prn} of a comparability graph and the dimension of its induced poset (cf. \cite[Corollary 2]{khyodeno2}), Theorem \ref{3-per-rep} gives us an alternative proof for \cite[Theorem 22]{split_orders2004}, as stated in the following corollary.

\begin{corollary}
	Let $G$ be a split comparability graph. The dimension of an induced poset of $G$ is at most three and the bound is tight. 
\end{corollary}

\end{document}